\documentclass[12pt]{birkjour}
\newtheorem{theorem}{\sc Theorem}[section]
\newtheorem{lemma}[theorem]{\sc Lemma}
\newtheorem{proposition}[theorem]{\sc Proposition}

\begin{document}
\title[Exponent of a finite group with an automorphism]
{Exponent of a finite group of odd order with an involutory automorphism}
\author{Sara Rodrigues}
\address{ Department of Mathematics, University of Brasilia,
	Brasilia-DF,   70910-900 Brazil }
\email{sararaissasr@hotmail.com}
\author{Pavel Shumyatsky}
\address{Department of Mathematics, University of Brasilia,
Brasilia-DF,   70910-900 Brazil }
\email{pavel2040@gmail.com}

\thanks{Supported by FINATEC and CNPq-Brazil}
\keywords{Finite Groups, Automorphisms, Rank, Exponent}
\subjclass{20D45}

\begin{abstract}
\noindent  
Let $G$ be a finite group of odd order admitting an involutory automorphism $\phi$. We obtain two results bounding the exponent of $[G,\phi]$. Denote by $G_{-\phi}$ the set $\{[g,\phi]\,\vert\, g\in G\}$ and by $G_{\phi}$ the centralizer of $\phi$, that is, the subgroup of fixed points of $\phi$. The obtained results are as follows.

\noindent 1. Assume that the subgroup $\langle x,y\rangle$ has derived length at most $d$ and $x^e=1$ for every $x,y\in G_{-\phi}$. Suppose that $G_\phi$ is nilpotent of class $c$. Then the exponent of $[G,\phi]$ is $(c,d,e)$-bounded.

\noindent  2. Assume that $G_\phi$ has rank $r$ and $x^e=1$ for each $x\in G_{-\phi}$. Then the exponent of $[G,\phi]$ is $(e,r)$-bounded.

\end{abstract}
\maketitle

\section{Introduction}

Let $G$ be a finite group of odd order admitting an involutory automorphism $\phi$. Here the term ``involutory automorphism" means an automorphism $\phi$ such that $\phi^2=1$. We denote by $G_{-\phi}$ the set $\{[g,\phi]\,\vert\, g\in G\}$ and by $G_{\phi}$ the centralizer of $\phi$, that is, the subgroup of fixed points of $\phi$. The subgroup generated by $G_{-\phi}$ will be denoted by $[G,\phi]$. It is well-known that $[G,\phi]$ is normal in $G$ and $\phi$ induces the trivial automorphism of $G/[G,\phi]$. The following theorem was proved in \cite{bebuja}:
\medskip

\noindent {\it Let $p$ be an odd prime and $G$ a finite $p$-group admitting an involutory automorphism $\phi$ such that $G_\phi$ is nilpotent of class $c$ and $x^p=1$ for each $x\in G_{-\phi}$. Suppose that the derived length of $G$ is at most $d$. Then the nilpotency class of $[G,\phi]$ is $(c,d,p)$-bounded.}
\medskip

Throughout this note we use the term ``$(a,b,c\dots)$-bounded" to mean ``bounded from above by some function depending only on the parameters $a,b,c\dots$". An immediate corollary of the above theorem is that under its hypotheses the exponent of $[G,\phi]$ is $(c,d,p)$-bounded as well. Recall that the exponent of a group $K$ is the minimal number $e$ such that $x^e=1$ for each $x\in K$. In this note we address the question whether one can relax the hypotheses of the theorem while still being able to conclude that the exponent of $[G,\phi]$ is bounded in terms of relevant parameters. Let $\langle X\rangle$ denote the subgroup generated by the set $X$. We obtain the following result.

\begin{theorem}\label{main} Let $c,d,e$ be nonnegative integers and $G$ a finite group of odd order admitting an involutory automorphism $\phi$ such that $G_\phi$ is nilpotent of class $c$ and $x^e=1$ for each $x\in G_{-\phi}$. Suppose that the subgroup $\langle x,y\rangle$ has derived length at most $d$ for every $x,y\in G_{-\phi}$. Then the exponent of $[G,\phi]$ is $(c,d,e)$-bounded.
\end{theorem}

We do not know whether or not all hypotheses in this theorem are necessary. It is conceivable that the exponent of $[G,\phi]$ in Theorem \ref{main} can be bounded only in terms of $c$ and $e$, or only in terms of $d$ and $e$. This seems to be a difficult problem.

Recall that a finite group $G$ is said to be of rank $r$, if $r$ is the least number such that any subgroup of $G$ is $r$-generator. Our next theorem shows that the exponent of $[G,\phi]$ can be bounded in terms of $e$ and the rank of $G_\phi$. In this case we do not even need to assume that $G_\phi$ is nilpotent.

\begin{theorem}\label{main2} Let $e$ and $r$ be nonnegative integers and $G$ a finite group of odd order admitting an involutory automorphism $\phi$ such that $G_\phi$ has rank $r$ and $x^e=1$ for each $x\in G_{-\phi}$. Then the exponent of $[G,\phi]$ is $(e,r)$-bounded.
\end{theorem}

We do not know if our results can somehow be generalized to the cases where $\phi$ is
not necessarily of order two, or the assumption that $(|G|,|\phi|)=1$ is dropped
from the hypotheses. 

Throughout the paper the Feit--Thompson Theorem \cite{fetho} is used without explicit references.

\section{Proof of Theorem \ref{main}}

The next lemma is a collection of well-known facts about involutory automorphisms. In the sequel we will frequently use it without any reference.

\begin{lemma}\label{1} If $G$ is a finite group of odd order admitting an involutory automorphism $\phi$, then
\begin{enumerate}
\item $G=G_\phi G_{-\phi}=G_{-\phi}G_\phi$, and each element $x\in G$ can be written uniquely in the form $x=gh$, where $g\in G_{-\phi}$ and $h\in G_{\phi}$;
\item If $N$ is any $\phi$-invariant normal subgroup
of $G$ we have $(G/N)_\phi=G_\phi N/N$, and
$(G/N)_{-\phi}=\{gN;g\in G_{-\phi}\}$;
\item If $N$ is a $\phi$-invariant normal subgroup of
$G$ such that either $N=N_{-\phi}$ or $N=N_{\phi}$, then $[G,\phi]$ centralizes $N$;
\item The normal closure of $G_{\phi}$ contains $G'$;
\item $G_{\phi}$ normalizes the set $G_{-\phi}$.
\end{enumerate}
\end{lemma}

\begin{proposition} \label{5} Let $G$ be a finite group of odd order admitting an involutory automorphism $\phi$ and assume that $G$ is soluble with derived length $d$. Suppose that $G=[G,\phi]$ and $(G_{-\phi})^{e}=1$. Then $G$ has $(d,e)$-bounded exponent.
\end{proposition}
\begin{proof} We use induction on $d$. If $d=1$, the result is obvious, so we assume that $d\geq2$. Let $M=G^{(d-1)}$ be the last nontrivial term of the derived series of $G$. By induction, the exponent of $G/M$ is $(d,e)$-bounded. Since $M=M_{-\phi}M_\phi$ and $(M_{-\phi})^{e}=1$, taking into account that $M$ is abelian, we deduce that $M^e\leq M_\phi$. In view of Lemma \ref{1}.3, we conclude that $M^e\leq Z(G)$. Hence, the exponent of $G/Z(G)$ is $(d,e)$-bounded. A theorem of Mann \cite{mann} now guarantees that $G'$ has $(d,e)$-bounded exponent. Since $G$ is generated by elements of order dividing $e$, the result follows.
\end{proof}

\begin{lemma}\label{9} Let $G$ be a finite group of odd order with an involutory automorphism $\phi$ such that $G=[G,\phi]$. Let $S$ be the set of those elements $h\in G_\phi$ for which there exist $x,y\in G_{-\phi}$ such that $h\in\langle x,y\rangle$. Then $G_{\phi}=\langle S\rangle$.
\end{lemma}
\begin{proof} Set $H=\langle S\rangle$. Obviously, $H\leq G_\phi$. So we need to prove that $G_\phi\leq H$. Choose $h\in G_{\phi}$. Since $G=[G,\phi]$, we can write $h=g_1\cdots g_n$ with $g_i\in G_{-\phi}$. We wish to prove that $h\in H$. This will be shown by induction on $n$. If $n\leq2$, then obviously $h\in H$, so assume that $n\geq 3$. Let $K=\langle g_{n-1},g_{n}\rangle$ and note that $K$ is $\phi$-invariant. Write $g_{n-1}g_n=g_0h_0$ where $g_{0}\in K_{-\phi}$ and $h_{0}\in K_{\phi}$. Observe that $K_\phi\leq H$ and therefore $h_0\in H$. We have $h=g_1\cdots g_{n-2}g_0h_0$ and $hh_0^{-1}=g_1\cdots g_{n-2}g_0$. By induction we get $hh_0^{-1}\in H$, whence $h\in H$. Therefore, we conclude that $H=G_{\phi}$.      
\end{proof}	

The following theorem was proved in \cite{shu} (see also \cite {shu11}). Its proof is based on a Lie-theoretical result of Zelmanov \cite{ze}.

\begin{theorem}\label{3.1} Let $e$ be a positive integer and $G$ a finite group of odd order admitting an involutory automorphism $\phi$ such that all elements in $G_{\phi}\cup G_{-\phi}$ have order dividing $e$. Then the exponent of $G$ is $e$-bounded.
\end{theorem}

Recall that if $T$ is a nilpotent group of class $c$ generated by elements of order dividing $e$, then the exponent of $T$ divides $e^c$ (see for example \cite[Corollary 2.5.4]{khu}). We are now ready to prove Theorem \ref{main}.

\begin{proof}[Proof of Theorem \ref{main}.] In view of Theorem \ref{3.1}, it is sufficient to show that $G_\phi$ has $(c,d,e)$-bounded exponent. Without loss of generality we can assume that $G=[G,\phi]$. Let $S$ be the set of those elements $h\in G_\phi$ for which there exist $x,y\in G_{-\phi}$ such that $h\in\langle x,y\rangle$. Lemma \ref{9} says that $G_{\phi}=\langle S\rangle$. Combining Proposition \ref{5} with our hypotheses we conclude that the elements in $S$ have $(c,d,e)$-bounded  order. Thus, $G_\phi$ is a nilpotent group of class $c$ generated by elements of bounded order. Hence, the exponent of $G_\phi$ is $\{c,d,e\}$-bounded, as required.   	
\end{proof}	

\section{Proof of Theorem \ref{main2}}
In what follows we write $r(K)$ for the rank of a group $K$. We start this section with a theorem established in \cite{1998}. Its proof uses the theory of powerful $p$-groups, due to Lubotzky and Mann \cite{luma}.

\begin{theorem}\label{98} Let $G$ be a finite group of odd order admitting an involutory automorphism $\phi$ such that $r(G_\phi)=r$. Then $r([G,\phi]')$ is $r$-bounded.
\end{theorem}

Let $p$ be a prime and $K$ a $p$-group. Suppose that a subgroup of $K$ generated by a subset $X\subseteq K$ can be generated by $m$ elements. In the sequel we will use without explicit references the well-known observation that $\langle X\rangle$ can be generated by $m$ elements from the set $X$. This fact can be easily deduced from the Burnside Basis Theorem (see \cite[Theorem 5.3.2]{rob}). Another well-known fact that we will require is the following lemma.

\begin{lemma}\label{err} The order of a finite group $G$ of exponent $e$ and rank $r$ is $(e,r)$-bounded.
\end{lemma}
\begin{proof} Indeed, if $G$ is a $p$-group the result is immediate from \cite[Corollary 11.21]{khu2}. Since the order of $G$ is the product of the orders of its Sylow subgroups, the general case follows from the observation that the prime divisors of the order of $G$ are divisors of $e$.
\end{proof}

\begin{proof}[Proof of Theorem \ref{main2}.] Without loss of generality we may assume that $G=[G,\phi]$. Let $r'=r(G')$. Theorem \ref{98} tells us that $r'$ is $r$-bounded. It follows that the Fitting height of $G$ is $r$-bounded as well (see for example \cite[Lemma 2.4]{glasgo}). Arguing by induction on the Fitting height of $G'$ we can assume that the exponent of $G/F(G')$ is $(e,r)$-bounded. Thus, there is an $(e,r)$-bounded positive integer $f$ such that $G^f\leq F(G')$. Set $M=G^f$. Note that both the rank and the exponent of $(G/M)_\phi$ are $(e,r)$-bounded. We conclude that the order of $(G/M)_\phi$ is $(e,r)$-bounded, too (Lemma \ref{err}). Now a theorem of Hartley and Meixner \cite{hame} says that $G/M$ has a subgroup of bounded index which is nilpotent of class at most two. We conclude that the derived length of $G/M$ is $(e,r)$-bounded.

Assume that $M$ is a $p$-group for some prime $p$. Let $A$ be a maximal normal abelian subgroup of $M$. Note that $A=C_M(A)$ (see \cite[Theorem 5.2.3]{rob}). Denote by $B$ the minimal characteristic subgroup of $M$ containing $A$. Since $B$ is a product of at most $r'$ subgroups of the form $A^\alpha$, where $\alpha\in Aut\, G$, it follows that the nilpotency class of $B$ is at most $r'$. Of course, $C_M(B)\leq B$. Let $D$ be the minimal normal subgroup of $G$ containing $B_{-\phi}$. This is a subgroup of class at most $r'$ generated by elements of order dividing $e$. It follows that the exponent of $D$ is $(e,r)$-bounded. Taking into account that the rank of $D$ is at most $r'$ we conclude that the order of $D$ is $(e,r)$-bounded, too. It follows that there is an $(e,r)$-bounded number $j$ such that $D$ is contained in $Z_j(M)$, the $j$th term of the upper central series of $M$. Observe that $\phi$ acts trivially on $B/D$. In view of Lemma \ref{1}.3 we conclude that $B/D\leq Z(G/D)$. Therefore $B$ is contained in $Z_{j+1}(M)$. It follows that the $(j+1)$th term of the lower central series of $M$ centralizes $B$. Recall that $C_M(B)\leq B$. Hence the $(j+1)$th term of the lower central series of $M$ is contained in $B$. Taking into account that the nilpotency class of $B$ is at most $r'$ we deduce that the derived length of $M$ is $(e,r)$-bounded. Since also the derived length of $G/M$ is $(e,r)$-bounded, it follows that the derived length of $G$ is $(e,r)$-bounded. By Proposition \ref{5}, the exponent of $G$ is $(e,r)$-bounded.

Thus, the theorem is proved in the particular case where $M$ is a $p$-group. Note that in that case the bound $e_0$ on the exponent of $G$ does not depend on the prime $p$. In general $M$ is a direct product of its Sylow $p$-subgroups. The above paragraph shows that the exponent of $G/O_{p'}(M)$ divides $e_0$ for each prime divisor $p$ of the order of $M$. Since $\bigcap_{p\in\pi(M)}O_{p'}(M)=1$, we conclude that the exponent of $G$ divides $e_0$. The proof is now complete.
\end{proof}

\end{document}